\documentclass[12pt]{article}
\usepackage[margin=1in]{geometry} 
\usepackage{graphicx} 
\usepackage{amsmath} 
\usepackage{amsfonts} 
\usepackage{amsthm} 
\usepackage{url}
\usepackage{authblk}
\usepackage[utf8]{inputenc}
\usepackage[T1]{fontenc}
\newtheorem{thm}{Theorem}[section]
\newtheorem{lem}[thm]{Lemma}
\newtheorem{prop}[thm]{Definition}
\newtheorem{cor}[thm]{Corollary}
\newtheorem{conj}[thm]{Conjecture}
\newtheorem{notation}[thm]{Notation}
\theoremstyle{rul} \newtheorem{rul}{Rule}
\newcommand{\E}{\mathcal{E}}

\newcommand{\floor}[1]{\left\lfloor{#1}\right\rfloor}
\newcommand{\nat}{\mathbb{N}}
\begin{document}
\nocite{*}
\title{New Upper and Lower Bounds on the Rado Numbers}
\author{William Gasarch, Russel Moriarty \\ 
					Department of Computer Science,\\
					The University of Maryland at College Park \and Nithin Tumma \\
												Port Huron Northern High School}

\date{}
\maketitle
\begin{abstract}
If $\E$ is a linear homogenous equation and $c\in\nat$ then the Rado number $R_c(\E)$ is the
least $N$ so that any $c$-coloring of the positive integers from $1$ to $N$ contains a monochromatic solution. 
Rado characterized for which $\E$ $R_c(\E)$ always exists.  The original proof of Rado's theorem gave
enormous bounds on $R_c(\E)$ (when it existed). In this paper we establish better upper bounds, and some
lower bounds, for $R_c(\E)$ for some $c$ and $\E$. 
In the appendix we use some of our theorems, and ideas from a probabilistic SAT solver,
to find many new Rado Numbers.
\end{abstract}

\section{Introduction}

\begin{notation}
If $n\in\nat$ then $[n]$ is the set $\{1,\ldots,n\}$.
\end{notation}

Let $\E$ be a linear homogenous equation and $N\in \nat$.
If you color $c$-color $[N]$ you may or may not get a monochromatic
solution to $\E$.

\begin{prop} 
If $\E$ is a linear homogenous equation.
\begin{enumerate}
\item
Let $c\in\nat$.  $R_c(\E)$ is the least positive integer $N$ (if it exists) 
such that any $c$-coloring of $[N]$ contains a monochromatic solution to $\E$. 
If we do not include the subscript then it is assumed to be 2.
\item
$\E$ is {\it $c$-regular} if $R_c(\E)$ exists.
\item
$\E$ is {\it regular} if, for all $c$,  $\E$ is $c$-regular.
\end{enumerate}
\end{prop}

In 1916 Schur proved that the equation $x + y = z$ is regular \cite{Schur}. 
In 1933 Rado, a graduate student of Schur's, determined exactly which systems of equations are regular~\cite{Rado}. 
His proof used an extension of van der Waerden's theorem and
hence lead to large bounds on $R_c(\E)$.
We present his theorem for single equations:

\begin{thm}[Rado's Theorem] 
The equation $a_1x_1 + a_2x_2 + a_3x_3 + \cdots + a_nx_n = 0$ 
is regular if and only if there exists a subset $I \in [n]$ 
such that ${\sum_{i \in I}{a_i}} = 0$. 
\end{thm}
For a proof of Rado's theorem theorem consult the books by Graham, Rothchild, and Spencer~\cite{RBook},
Landman and Robertson~\cite{RamseyInts} or the free on-line book of Gasarch, Kruskal, Parrish~\cite{VDWbook}.

It is an open problem in Ramsey Theory to find better upper bounds on the Rado Numbers.
To date 2-color Rado numbers have only been determined for a few classes of equations. 
In this paper we prove several theorems that give much better upper bounds on $R_c(\E)$ for
several $c$ and $\E$. We also have some computational results in the appendix.

Previous results have determined the Rado numbers for some classes of equations, completely characterized the 2-color Rado numbers for equations of the form $a(x+y) = bz$ \cite{HarRado}, while Robertson and Myers gave results and conjectures for four variable equations of the form of $x + y + kz = jw$ \cite{2color}. Here we examine the case $a(x-y) = bz$ and $x + ay = abz$, bounding the two color Rado numbers of both. Additionally some four variable equations are considered, and their Rado numbers proven. Furthermore, a result of Rado is extended, showing that all non-trivial two variable equations are not 2-regular. A new proof of Rado's single equation theorem is given, providing better bounds on Rado numbers in certain cases. Additionally, a probabilistic method approach is provided which gives lower bounds on Rado numbers of equations with an arbitrary number of colors. We conclude with our algorithm for computing Rado numbers, tables of computed 2 and 3-color Rado numbers, and conjectures that follow.

%
%
%
%
%
%
%
%
%
%
%
%
%
%
%
%
%
%
%
%
%
%
%
%
%
%
\section{Summary of Results}

We list our main results:

\begin{enumerate}
\item
Results on 2-coloring
\begin{enumerate}
\item
 $R_2(x-y = bz) = b^2 + 3b + 1$ 
\item     
$R_2(a(x-y) = bz) = a^2 $ , $a > b$ 
\item
$R_2(a(x-y) = bz) \geq b^2 + b + 1$ , $a \leq b$ 
\item
 $R_2(x + ay = abz) \geq a^2$  
\item
$R_2(px + a_1x_1 + a_2x_2 + a_3x_3 + \cdots + a_nx_n = abz) \geq a^2$ for $p$ relatively prime to $a$ and $a_i \equiv 0 \bmod{a}$. 
\item
$R_2(x + ay = 2az) = a^2$ 
\item
$R_2(x + y + az = (a+1)w) = 5$  , $a > 3$ 
\item
$R_2(2x + 2y + az = (a+3)w) = 10$  , $a > 24$ 
\item 
$R_2(3x + 3y + az = (a+5)w) = 15$  , $a \ge 30$  
\item
$R_2(ax = bz) = \infty$, $a \neq b$
\end{enumerate}
\item
Results on $c$-coloring
\begin{enumerate}
\item
$R_c(\E) \leq m^2 + 3m + 1$, where $\E$ is an equation $a_1x_1 + a_2x_2 + \cdots a_nx_n = 0$ that includes $I \subset [1, n]$ where $\sum_{i \in I} a_i = 0$ and a $q \in I$ such that $q$ divides $ \sum_{i \notin I} a_i$.
\item
$R_c(x-y = az) > \frac{\sqrt{(a+c-1)^2 + 8c^2(a^2+a)} -a -c +1}{2}$ 
\item
$R_c(a(x-y) = bz) \geq \frac{bc^3}{e(b+3)} + \frac{2b+3}{b+3}$
\end{enumerate}
\end{enumerate}

\section{Bounds on $2$-Color Rado Numbers}
We obtain new upper and lower bounds of 2-color Rado numbers for several classes of equations. 
We also obtain a new proof of Rado's single equation theorem in the $c=2$ case which leads to better
upper bounds for some Rado Numbers.

\subsection{$2$-Color Rado Numbers for $a(x-y) = bz$.}
In this section we characterize the Rado numbers of equations of the form $a(x-y) = bz$.

\begin{thm}
$R(x-y = bz) = b^2 + 3b + 1$
\end{thm}

\begin{proof} [\textbf{Lower Bound}]
We show a coloring of $[1, b^2 + 3b ]$ lacking a monochromatic solution to $x - y = bz$. Consider the coloring defined by $RRRR\ldots$ (Total of $b$ R's) followed by $BBBB\ldots$ (Total of $b^2+b$ $B$'s) followed by $RRRR\ldots$ (Total of $b$ $R$'s). It is shown that this coloring does not admit a monochromatic solution to $x-y=bz$. If z is $R$, and in $[b^2 + 2b + 1, b^2 + 3b]$, $bz \geq b^3 + 2b^2 + b$. But $x - y\leq b^2 + 3b - 1 < bz$ for $b > 1$. So, if $z$ is $R$ it must be in $[1, b]$. If $x$ is in $[b^2 + 2b + 1, b^2 + 3b]$ and $y$ is in $[1, b]$, $x - y \geq b^2 + 2b > bz$ because $bz \leq b^2$. On the other hand, if $x$ and $y$ are in $[1, b]$, $x-y \leq b-1 < bz$ because $bz \geq b$. If $z$ is $B$, $x - y \leq b^2 + b -1 < bz$, because $bz \geq b^2 + b$. So, there can be no monochromatic solution of $x - y =bz$ under this coloring of $[1, b^2 + 3b]$. 
\end{proof}

\begin{proof}[\textbf{Upper Bound}]
We show that for all colorings COL there must exist a monochromatic solution to $x-y = bz$ in $[1, b^2 + 3b + 1]$. We notice that $(q + bd, q, d)$ is a solution to $x-y = bz$ for any $q, d \in \mathbb{N}$. Note that $COL(q) \neq COL(q + qb)$ due to the solution $(q + qb, q, q)$. We can assume that $1 \in R$, so $(b+1) \in B$. Then, $b^2 + 2b + 1 \in R$. Because $(b^2 + 3b + 1, b^2 + 2b + 1, 1)$ is a solution, $b^2 + 3b + 1 \in B$.

\textbf{Case: $2 \in R$.} 
$b + 2 \in B$, and because $(b^2 + 3b + 1, b+1 ,b+2)$ is a solution, $b+2 \in R$, giving a contradiction. So $2 \in B$

\textbf{Case: $3 \in R$.} 
Because $(3b + 1, b+1, 2)$ is a solution, $3b+1 \in R$. But, since $(3b + 1, 1, 3)$ is a solution, $3b+1 \in B$, giving a contradiction. Thus, $3 \in B$. 

$(3b + 1, b + 1, 2)$ is a solution, so $3b +1 \in R$. Because $(4b + 1, b+1, 3)$ is also a solution, $4b + 1 \in R$. But $(4b+1, 3b+1, 1)$ is a solution, so $4b+1 \in B$, giving a contradiction. 
So any coloring of $[1, b^2 + 3b + 1]$, with $b \in \mathbb{N}$ must contain a monochromatic solution to $x - y = az$.
\end{proof}

\begin{thm}
$R(a(x-y) = bz) = a^2$ for $a > b$.
\end{thm}

\begin{proof}[\textbf{Lower Bound}]
We show a coloring of $[1, a^2 - 1]$ lacking a monochromatic coloring to $a(x-y) = bz$ for $a > b$. 
Consider the coloring defined by 
$$
\chi(x) = 
\begin{cases}
1 & \text{when $x \equiv 0\bmod{a}$} \\
0 & \text{otherwise}
\end{cases}
$$
We may assume that $a$ and $b$ are relatively prime by dividing common factors as necessary. Taking the equation mod $a$ we see that $bz \equiv 0\bmod{a}$. So, $z \equiv 0 \bmod{a}$. For a monochromatic solution under $\chi(x)$ to exist, $x,y \equiv 0\bmod{a}$. Let $z = na$, with $n \geq 1$. Then, rewrite the equation as 
$$
a(x-y) = bna
$$ 
Divide the equation by $a$ to get
$$
x - y = nb
$$
Reduce $\mod{a}$ to see that $nb \equiv 0\bmod{a}$. Because $a$ and $b$ are relatively prime, $n \equiv 0\bmod{a}$. Since $n \geq 1$, $n = ma$ with $m \geq 1$. Then $z = na = ma^2 \geq a^2$. Thus, no monochromatic solution to $a(x-y) = bz$ exists under $\chi(x)$ in $[1, a^2-1]$.
\end{proof}

\begin{proof}[\textbf{Upper Bound}]
We show for all colorings COL of $[1, a^2]$ there must exist a monochromatic coloring to $a(x-y) = bz$ with $a > b$. Assume, for contradiction, that there exists a coloring of $[1, a^2]$ without a monochromatic solution to $a(x-y) = bz$. Without loss of generality, let $a \in R$. Considering the solution $(a, a-b, a)$, we see that $a-b \in B$. From $(a+b, a, a)$, $a+b \in B$ as well. Then $2a \in R$ so $a+b, a-b, 2a$ is not monochromatic. We see that from $(2a, 2(a-b), 2a)$, $2(a-b) \in B$. From $(2a + b, 2a, a)$, we have that $2a + b \in B$. Then from $(2a + b, 2(a-b), 3a)$ we see that $3a \in R$. Continuing in this fashion, we see that $na \in R$ for $1 \leq n \leq a$. But $(a^2, (a-b)a, a^2)$ is a monochromatic solution, a contradiction.
\end{proof}

\begin{thm}
$R(a(x-y) = bz) \geq b^2 + b + 1 $ for $a \leq b$.
\end{thm}

\begin{proof}[\textbf{Lower Bound}]
We show a coloring of $[1, b^2 + b]$ lacking a monochromatic coloring to $a(x-y) = bz$ for $a \leq b$. 
Consider the coloring defined by 
$$
\chi_1(x) = 
\begin{cases}
1 & \text{when $1 \leq x \leq ab$ AND $x \equiv 0\bmod{a}$} \\
0 & \text{when $x \not\equiv 0\bmod{a}$} \\
0 & \text{when $ab + 1 \leq x \leq b^2 + b$} \\
\end{cases}
$$

We may assume that $a$ is relatively prime to $b$ by dividing common factors as necessary. Now, reducing the equation $\mod{a}$ we see that $bz \equiv 0\bmod{a}$. So, $z \equiv 0\bmod{a}$, and we can let $z = na$, with $n \geq 1$. First we show that $COL(z) \neq 1$. If $COL(z) = 1$, then $x-y \leq a(b-1)$. But then $x - y = ma$ with $1 \leq m \leq b-1$. However, $a(x-y) = ma^2$, and because $ma^2$ is not divisible by $b$, $a(x-y)$ cannot equal $bz$. So $COL(z) = 0$ and $z \geq a(b+1)$. For $a(x-y) = bz$ to hold, $x - y \geq b^2 + b$, but $x - y \leq b^2 + b - 1$. So, there does not exist a monochromatic solution to $a(x-y) = bz$ under $\chi_1(x)$. 
\end{proof}

Based on computed $2$-color Rado numbers we make the following conjecture:
\begin{conj}
$R(a(x-y) = bz) = b^2 + b + 1 $ for $a \leq b$.
\end{conj}

\subsection{Rado Numbers for $x + ay = abz$.}
We give results for equations of the form $x + ay = abz$, proving a general lower bound applicable to equations with arbitrarily many variables. At the end of the section a few conjectures are posed based on evidence from computed 2 and 3-color Rado.
\begin{thm}
$R(px + a_1x_1 + a_2x_2 + a_3x_3 + \cdots + a_nx_n = abz) \geq a^2$ for $p$ relatively prime to $a$ and $a_i \equiv 0 \bmod{a}$. 
\end{thm}
\begin{proof} 
Consider the coloring $\chi(x)$ defined by
$$
\chi(x) = 
\begin{cases}
1 & \text{when $x \equiv 0\bmod{a}$} \\
0 & \text{otherwise}
\end{cases}
$$
Reduce the equation mod $a$ to get $px \equiv 0 \bmod{a}$. For a monochromatic solution under $\chi(x)$, $x_i, z$ must be $\equiv$ $0 \bmod{a}$. Because $p$ is relatively prime to $a$, $x \equiv 0 \bmod{a}$. Let $x = na$, and rewrite the equation as $$pna + a_1x_1 + a_2x_2 + a_3x_3 + \cdots + a_nx_n= abz$$
Let $\frac{a_i}{a} = c_i$ and divide the equation by $a$ to get $$pn + c_1x_1 + c_2x_2 + \cdots + c_nx_n= bz.$$ Reduce this mod $a$ to find $pn \equiv 0 \bmod{a}$. Because $p$ is relatively prime to $a$, $n \equiv 0 \bmod{a}$. So $n = ma$, with $m \geq 1$, and $x = na = ma^2$. Thus, $x$ $\geq a^2$ and a monochromatic solution to $px + a_1x_1 + a_2x_2 + a_3x_3 + \cdots = abz$ cannot exist under $\chi(x)$ in $[1, a^2 - 1]$. 
\end{proof}

\begin{thm} 
$R(x + ay = 2az) = a^2$. 
\end{thm}
\begin{proof}[\textbf{Lower Bound}]
The lower bound is given by the theorem above. 
\end{proof}
\begin{proof}[\textbf{Upper Bound}]
We show that any coloring $COL$ of $[1, a^2]$ contains a monochromatic solution to $x + ay = 2az$.
We show that if $a \equiv 1 \bmod{2}$, we are guaranteed a monochromatic solution in $[1, a^2]$. Because $(a^2, a, a)$ is a solution to the above equation, $COL(a^2) \neq COL(a)$. From the solution $(a, a, 1)$, we see that $COL(a) \neq COL(1)$, implying that $COL(a^2) = COL(1)$. $(a^2, a^2, \frac{a^2 + a}{2})$ is another solution to the equation, so $COL(a^2) \neq COL(\frac{a^2 + a}{2})$. Because $(\frac{a^2 + a}{2}, \frac{a+1}{2}, \frac{a+1}{2})$ is also a solution to the equation, $COL(\frac{a^2 + a}{2}) \neq COL(\frac{a+1}{2})$, implying that $COL(a^2) = COL(\frac{a+1}{2})$. However, we now have a monochromatic solution: $(a^2, 1, \frac{a+1}{2})$.
For the case $a \equiv 0 \bmod{2}$, we see that we can write $a$ as $m2^i$ where $m \equiv 1 \bmod{2}$. Then, the equation becomes $x = m2^iy = m2^{i+1}z$. Reducing mod $2^i$, we see that $x \equiv 0 \bmod {2^i}$. So we can divide the equation by $2^i$, giving $x + my = 2mz$. This is equivalent to the case of $a \equiv 0 \bmod{2}$, which was proved above. 
\end{proof}

Using a similar argument, we find a lower bound for the $2$-color Rado numbers of equations of the form $x + a^ny = a^nbz$, with $a, b \in \mathbb{N}$.

From computed values of $2$-color Rado numbers, we present the following conjecture.
\begin{conj}
$R(x + ay = abz) = a^2$ with $a \geq 2b -1$.
\end{conj}

\subsection{Some $2$-Color Rado Numbers for $b(x + y) + az = (a + 2b - 1)w$.}
In this section we give some 2-color Rado numbers for some four variable equations.

\begin{thm}
$R(x + y + az = (a+1)w) = 5$ for $a > 3$.
\end{thm}

\begin{proof}[\textbf{Lower Bound}]
It is easy to check that the coloring $RBRB$ contains no monochromatic solutions to $x + y + az = (a+1)w$ in $[1, 4]$.
\end{proof}

\begin{proof}[\textbf{Upper Bound}]
Assume, for contradiction, that there exists a coloring of $[1, 5]$ without a monochromatic solution to $x - y + az = (a+1)w$. Without loss of generality let $1 \in R$. From the solution $(1, 1, 2, 2)$, we see that $2 \in B$. From $(2, 2, 4, 4)$ we see that $4 \in R$. Then, $3 \in B$ because of the solution $(3, 1, 4, 4)$, and $5 \in R$ from $(2, 3, 5, 5)$. But then we have the monochromatic solution $(4, 1, 5, 5)$, a contradiction. 
\end{proof}

\begin{thm}
$R(2x + 2y + az = (a+3)w) = 10$ for $a > 24$
\end{thm}

\begin{proof}[\textbf{Lower Bound}]
It is easy to see that the coloring $RBBRBRRBR$ does not contain a monochromatic solution to $2x + 2y + az = (a+3)w$ in $[1, 9]$. 
\end{proof}

\begin{proof}[\textbf{Upper Bound}]
Assume, for contradiction, that there exists a coloring of $[1, 10]$ that does not contain a monochromatic solution to $2x + 2y + az = (a+3)w$. Without loss of generality let $1 \in R$. From $(1, 2, 2, 2)$ we see that $2 \in B$. From $(2, 4, 4, 4)$, $4 \in R$. $5 \in B$ so $(1, 5, 4, 4)$ is not monochromatic. From $(3, 3, 4, 4)$ we see that $3 \in B$, and from $(3, 6, 6, 6)$, $6 \in R$. $8 \in B$ so $(4, 8, 8, 8)$ is not monochromatic, and $7 \in R$ because of the solution $(5, 7, 8, 8)$. We see that $10 \in R$ from $(2, 10 , 8, 8)$ and that $9 \in B$ from $(9, 6, 10, 10)$. But we have the monochromatic solution $(9, 3, 8, 8)$, a contradiction.
\end{proof} 

\begin{thm}
$R(3x + 3y + az = (a+5)w) = 15$ for $a \ge 30$.
\end{thm}

\begin{proof}[\textbf{Lower Bound}]
It is easy to see that the coloring $RBRBBRBRBRRBRB$ admits no monochromatic solutions to $3x + 3y + az = (a+5)w$ in $[1, 14]$. 
\end{proof}

As the Upper Bound is very similar to the previous cases, it is left to the reader.

It is tempting to attempt to extend the results to the general case $b(x + y) + az = (a + 2b - 1)z$, but we quickly see that these do not seem to follow the $R_2(\E) = 5b$ that seems to hold for $1 \leq b \leq 3$.

\subsection{When is an Equation $2$-Regular?}
Here we expand on a result of Rado with respect to the $2$-regularity of linear equations. Recall that an equation is $k$-regular if there exists an $n \in \mathbb{N}$ such that all $k$-colorings of $[1, n]$ contain a monochromatic solution. Rado proved the following result \cite{Rado}. The following proof was included in \cite{Off}, we include it for completeness. 
\begin{thm}
All equations in three or more variables with positive and negative coefficients are 2-regular. 
\end{thm}
\begin{proof}
Let $\sum_{i=1}^k \alpha_i x_i = \sum_{i=1}^{\ell} \beta_i y_i$
be our equation, where $k \geq 2$, $\ell \geq 1$,
$\alpha_i \in \mathbb{Z}^+$ for $1 \leq i \leq k$,
and $\beta_i \in \mathbb{Z}^+$ for $1 \leq i \leq \ell$.
By setting $x=x_1=x_2=\cdots=x_{k-1}$, $y=x_{k}$, and
$z=y_1=y_2=\cdots=y_{\ell}$, we  may consider solutions to
$$
ax+by=cz,
$$
where $a=\sum_{i=1}^{k-1} \alpha_i$, $b=\alpha_k$, and
$c = \sum_{i=1}^{\ell} \beta_i$.  We will
denote $ax+by=cz$ by $\E$.

Let $m = \mathrm{lcm}\left(\frac{a}{\gcd(a,b)},
\frac{c}{\gcd(b,c)}\right)$.  Let $(x_0,y_0,z_0)$
be the solution to $\E$ 
with $\max(x,y,z)$ a minimum, where the maximum
is taken over all solutions of positive integers
to $\E$.  Let $A=\max(x_0,y_0,z_0)$.

Assume, for a contradiction, that there exists
a $2$-coloring of $\mathbb{Z}^+$ with no
monochromatic solution to $\E$.
First, note that for any $n \in \mathbb{Z}^+$, 
the set $\{in: i=1,2,\dots,A\}$ cannot be monochromatic,
for otherwise $x=x_0n$, $y=y_0n$, and $z=z_0n$ is a monochromatic
solution, a contradiction.

Let $x=m$ so that $\frac{bx}{a}, \frac{bx}{c} \in \mathbb{Z}^+$.
Letting red and blue be our two colors, we may assume,
without loss of generality, that $x$ is red.
Let $y$ be the smallest number in
$\{im: i=1,2,\dots,A\}$ that is blue.
Say $y =\ell m$ so that $2 \leq  \ell \leq A$.

For some $n \in \mathbb{Z}^+$, we have that
 $z = \frac{b}{a} (y-x)n$ is
blue, otherwise $\{i \frac{b}{a}(y-x): i=1,2,\dots\}$
would be red, admitting a monochromatic
solution to $\E$. Then $w =
\frac{a}{c}z +
\frac{b}{c} y$ must be red, for otherwise $az+by=cw$ and
$z,y,$ and $w$ are all blue, a contradiction.
Since $x$ and $w$ are both red, we have that
$
q=\frac{c}{a}w - \frac{b}{a}x = \frac{b}{a} (y-x)(n+1)
$
must be blue, for otherwise $x,w,$ and $q$ give
a red solution to $\E$.
As a consequence, we see that
$
\left\{i\frac{b}{a}(y-x) : i=n,n+1,\dots\right\}
$
is monochromatic.  This gives us that
$
\left\{i\frac{b}{a} (y-x) n: i=1,2,\dots,A\right\}
$
is monochromatic, a contradiction.

\end{proof}

For another proof consult \cite{radoenglish}. Here we show that all two variable equations of the form $ax = by$ are not $2$-regular, excluding the trivial case $a = b$.

\begin{thm}
The equation $ax = by$ is not $2$-regular for $a \neq b$.
\end{thm}

\begin{proof}
The case $a = b$ is regular for all $k$, due to the trivial solution $(k, k)$. Without loss of generality assume that $b > a$. We may assume that $a$ and $b$ are relatively prime by dividing common factors as necessary. Now, clearly $a$ and $b$ must either both be positive or negative, else no solution $(x, y)$ can exist with $x, y \in \mathbb{N}$. We show a coloring $\chi(x)$ of $\mathbb{N}$ that does not contain a monochromatic solution to $ax = by$. Note that for all $j \in \mathbb{N}$, $(\frac{b}{a})^ i < j < (\frac{b}{a})^{i+1}$ for a unique value of $i$. Now let $v(j) = i$. Define the coloring 
$$
\chi(x) = 
\begin{cases}
1 & \text{if $v(x) \equiv 0\bmod{2}$} \\
0 & \text{if $v(x) \equiv 1\bmod{2}$ }
\end{cases}
$$
Writing $ax = by$ as $x = \frac{b}{a}y$ we see that, for a particular solution $(x, y)$, if $v(y) = k$ then $v(x) = k+1$. Under the coloring $\chi(x)$ there can be no monochromatic solutions to $ax = by$.
\end{proof}

It is interesting to note that the above theorem is valid for all of $\mathbb{Q}$, as opposed to the other theorems in this paper, in which we are only concerned with $\mathbb{N}$. 

\begin{section}{Better Bounds on the Rado Function}

We present an alternative proof of the single equation Rado's theorem which yields better
upper bounds in some cases.

Recall Van der Waerden's theorem:
\begin{thm}[Van der Waerden's Theorem]
For all $k,c \in \mathbb{N}$, there exists $W=W(k,c)$ such that for all $c$-colorings $\chi :[W] \rightarrow [c]$ there exist $a,d \in \mathbb{N}$ such that $\chi(a) = \chi(a+d) = \chi(a+2d) = \cdots = \chi(a+(k-1)d)$.
\end{thm}
This was first proven by van der Warden~\cite{VDW}. See the books by Graham, Rothchild, and Spencer~\cite{GRS},
Landman and Robertson~\cite{RamseyInts} or the free on-line book of Gasarch, Kruskal, Parrish~\cite{VDWbook}
for the proof in English. 

This proof gives enormous upper bounds on the numbers $W(k,c)$ that are not primitive recursive.
Shelah~\cite{VDWs} gave an alternative proof that yields primitive recursive upper bounds.
All of the proofs noted above are elementary. Gowers~\cite{Gowers} 
provided a non-elementary proof that yields much better better, though still huge, bounds.

The following variant of Van der Waerden's Theorem is used to 
prove Rado's theorem and determine better bounds on Rado numbers. 

\begin{thm}[VDW Variant]\label{th:vdwv}
For all $k,l,m,c \in \mathbb{N}$, there exists $U=U(k,l,m,c)$ such that for all $c$-colorings $\chi :[U] \rightarrow [c]$ there exist $a,d \in \mathbb{N}$ such that
\[
\chi(a) = \chi(a+md) = \chi(a+2md) = \cdots = \chi(a+(k-1)md) = \chi(ld)
\]
\end{thm}

\begin{proof}
The proof is by induction on $c$. Clearly for all $k,l,m$, we have that $U(k,l,m,1) = \text{max}\{1+(k-1)m,l\}$.
For the induction step we assume $U(k,l,m,c-1)$ exists and use it to prove the existence of $U(k,l,m,c)$. Let $\chi$ be a $c$-coloring of $[W(k',c)]$, where $k' = (k-1)lmU(k,l,m,c-1)+1$. By the definition of $W(k',c)$, there exist $a,d$ such that
\[
\chi(a)=\chi(a+d)=\chi(a+2d)=\cdots=\chi(a+(k'-1)d)
\]
Without loss of generality assume this color is RED. This implies that for all $i \in [U(k,l,m,c-1)]$,
\[
\chi(a)=\chi(a+mid)=\chi(a+2mid)=\cdots=\chi(a+(k-1)mid)=\text{RED}
\]
Now there are two cases:

CASE 1: There exists $i \in [U(k,l,m,c-1)]$ such that $\chi(lid)=\text{RED}$. Therefore
\[
\chi(a)=\chi(a+mid)=\chi(a+2mid)=\cdots=\chi(a+(k-1)mid)=\chi(lid)=\text{RED}
\]
and we are done with this case.

CASE 2: For all $i \in [U(k,l,m,c-1)]$, $\chi(lid)\neq \text{RED}$. This implies that $\chi$ gives a $(c-1)$-coloring of $\{ild\}_{i\in [U(k,l,m,
c-1)]}$. By the definition of $U(k,l,m,c-1)$, there exist $a',d'$ such that
\[
\chi(a'ld)=\chi((a'+md')ld)=\chi((a'+2md')ld)=\cdots=\chi((a'+(k-1)md')ld) = \chi(ld'ld)
\]
Substituting $A=a'ld$ and $D=d'ld$ gives
\[
\chi(A)=\chi(A+mD)=\chi(A+2mD)=\cdots=\chi(A+(k-1)mD) = \chi(lD)
\]
and we are done.
\end{proof}

\begin{subsection}{Proof of Rado's Theorem}
We begin by presenting the lemma used in our proof of Rado's Theorem, then present Rado's Theorem itself.

\begin{lem}
For all $c,l \in \mathbb{N}$ and for all $m \neq 0 \in \mathbb{Z}$, 
there exists a $P=P(l,m,c) \in \mathbb{N}$ such that for all $c$-colorings of $[P]$, there exists $a,d \in \mathbb{N}$ such that $a,ld,a+md \in [P]$ are monochromatic.
\end{lem}

\begin{proof}
Let $U$ be the function from Theorem~\ref{th:vdwv}.
Let $P=U(2,l,m,c)$ and $\chi : [P] \rightarrow [c]$ be a coloring of $[P]$. By the definition of $U$, there exist $a,d$ such that $\chi(a)=\chi(a+md)=\chi(ld)$, which is exactly what we wanted to prove.
\end{proof}

We use the following notation for the Rado number in the remainder of this section:
\begin{prop}
The Rado number $R(a_1,\ldots,a_n;c)$ is the smallest $R$ such that for all $c$-colorings of $[1, R]$ there exists a monochromatic solution to $a_1x_1+\ldots+a_nx_n = 0$.
\end{prop}

\begin{thm}[Rado's Theorem]
For all $a_1,\ldots,a_n \in \mathbb{Z}$, if there exists an $I \subseteq [1, n]$ such that $\sum_{i\in I}a_i = 0$, then for all $c\geq 1$, $\exists R(a_1,\ldots,a_n;c)$ such that for all $c$-colorings of $[1, R]$ there exists a monochromatic solution to $a_1x_1+\cdots+a_nx_n = 0$, where each $x_i\in [1, R]$. 

\noindent $R(a_1,\ldots,a_n;c)$ satisfies the following upper bound:
$$
R(a_1, a_2, \ldots, a_n; c) \leq P\left(\frac{LCM\left(\sum_{i\notin I}a_i,a_q\right)}{\sum_{i\notin I}a_i},-\frac{LCM\left(\sum_{i\notin I}a_i,a_q\right)}{a_q},c\right)
$$
\end{thm}

\begin{proof}
Define
\[
s = \sum_{i\notin I}a_i
\]
Choose $q \in I$ such that $\left|LCM(s,a_q)\right|$ is minimal, and let $u=LCM(s,a_q)$, where $u$ is chosen to have the same sign as $s$.
We claim that if there exist positive integers $a$ and $d$ such that $a, \frac{ud}{s}, a-\frac{ud}{a_q} \in [R]$ are monochromatic, then there exists a monochromatic solution to the above equation. Namely,
\[
x_i = \left\{
\begin{array}{cl}
a-\frac{ud}{a_q} & \text{if } i = q \\
a & \text{if } i \neq q \in I\\
\frac{ud}{s} & \text{if } i \notin I
\end{array} \right.
\]
We can verify this as follows:
\begin{align*}
\sum_{i=1}^n a_ix_i &= \sum_{i\in I}a_ix_i + \sum_{i\notin I}a_ix_i\\
&= \sum_{i\in I}a_ia - a_q\frac{ud}{b} + \sum_{i\notin I}a_i\frac{ud}{s}\\
&= 0 - a_q\frac{ud}{a_q} + s\frac{ud}{s}\\
&= 0
\end{align*}

We apply Lemma 1 with $l=\frac{u}{s}$ and $m=-\frac{u}{a_q}$ to obtain an $R$ large enough to guarantee the existence of a monochromatic triple $a,ld,a+md \in [R]$. Since $u$ was chosen to have the same sign as $s$, $l$ is guaranteed to be positive in our application of Lemma 1. If $a_q$ also has the same sign as $s$ then $m<0$, whereas if $a_q$ and $s$ have opposite signs then $m>0$.

Formally, we have shown the following:
\begin{align*}
R(a_1,\ldots,a_n;c) &\leq P\left(\frac{u}{s},-\frac{u}{a_q},c\right) \\
&= P\left(\frac{LCM\left(s,a_q\right)}{s},-\frac{LCM\left(s,a_q\right)}{a_q},c\right) \\
&= P\left(\frac{LCM\left(\sum_{i\notin I}a_i,a_q\right)}{\sum_{i\notin I}a_i},-\frac{LCM\left(\sum_{i\notin I}a_i,a_q\right)}{a_q},c\right)
\end{align*}
\end{proof}

The VDW proof of Lemma 1 gives the following upper bound on Rado numbers:
\begin{align*}
R(a_1,\ldots,a_n;c) &\leq P\left(\frac{LCM\left(\sum_{i\notin I}a_i,a_q\right)}{\sum_{i\notin I}a_i},-\frac{LCM\left(\sum_{i\notin I}a_i,a_q\right)}{a_q},c\right) \\
&= P\left(\frac{u}{s}, -\frac{u}{a_q}, c\right) \\
\end{align*}
\end{subsection}

\begin{subsection}{Quadratic Upper Bound on $R_2(\E)$ in a Special Case}
This section deals with the class of equations where, after forming $I \subseteq [n]$ such that $\sum_{i\in I}a_i=0$, there exists a $q\in I$ such that $a_q$ divides $s=\sum_{i\notin I}a_i$. For equations that fall into this category, $u=LCM(s,a_q)=s$. Therefore in the application of Lemma 1, $l=\frac{u}{s}=1$ and $m=\frac{u}{a_q}$. In addition we restrict our attention to the 2-color case.

\begin{lem}[$c=2,l=1$]
For all $m \in \mathbb{N}$ and for all $2$-colorings of $[1+3m+m^2]$ there exists a monochromatic triple $a,d,a+md \in [R]$.
\end{lem}

\begin{proof}
Our general approach is to do a case analysis of the potential colors that small numbers can take. There are two rules we use in this analysis. The first rule comes from taking $d=a$ in the above lemma.

\begin{rul}
For any $a \in \mathbb{N}$, $a+am$ cannot be the same color as $a$, otherwise we are done.
This is because $a,a,a+am$ would be a valid triple.
\end{rul}

The second rule is the more general case.

\begin{rul}
For any $a,d \in \mathbb{N}$ that share a color, neither $a+md$ nor $d+ma$ can be that same color, otherwise we are done.
\end{rul}

Without loss of generality assume 1 is colored RED. By Rule 1 that means $1+m$ must be BLUE, which means $(1+m)^2=1+2m+m^2$ must be RED. Applying Rule 2 we get that $(1+2m+m^2)+m(1)=1+3m+m^2$ must be BLUE.

CASE 1: 2 is RED. By Rule 2 that means $2+m$ must be BLUE, which by Rule 2 means $(1+m)+m(2+m)=1+3m+m^2$ must be RED. Since $1+3m+m^2$ must be either RED or BLUE, we are done.

CASE 2A: 2 is BLUE, 3 is RED. Since 2 is BLUE we can apply Rule 2 to it and $1+m$ to conclude that $(1+m)+m(2)=1+3m$ is RED. However since 1 and 3 are RED, Rule 2 implies that $1+3m$ is BLUE, so we are done.

CASE 2B: 2 is BLUE, 3 is BLUE. Rule 2 implies $(1+m)+m(2)=1+3m$ and $(1+m)+m(3)=1+4m$ must be RED, but applying Rule 2 to 1 and $1+3m$ implies $1+4m$ must be BLUE, so we are done with this case.
The result follows from the fact that $1+3m+m^2$ is greater than or equal to both $1+3m$ and $1+4m$ for $m\geq 1$.
\end{proof}

The next result follows directly from the above lemma.
\begin{thm}
$R_c(\E) \leq m^2 + 3m + 1$, where $\E$ is an equation $a_1x_1 + a_2x_2 + \cdots a_nx_n = 0$ that includes $I \subset [1, n]$ where $\sum_{i \in I} a_i = 0$ and a $q \in I$ such that $q$ divides $ \sum_{i \notin I} a_i$.
\end{thm}

\end{subsection}
\end{section}

\begin{section}{Lower Bounds on Rado Numbers with the  Probabilistic Method}
Here we present a new method of bounding Rado numbers, utilizing a probabilistic proof. With this approach, it is possible to obtain lower bounds of Rado numbers in arbitrarily many colors. To the best of our knowledge this is the first case of expressions for bounds of $R_c(\E)$ with $c > 2$. 

Let us define a few functions that will be used extensively throughout this section.

\begin{prop}
Let $\E$ be an equation in $j$ variables, then $\psi_{\E, i}(N)$ be the number of solutions to $\E$ in $[1, N]$, $(x_1, x_2, \ldots, x_j)$, with exactly $i$ distinct $x_k$.  
\end{prop}

\begin{prop}
Given an equation $\E$ in $j$ variables, let $\psi_N(\E)$, expressed as a function of $N$, give the number of integral solutions to $\E$ in $[1, N]$.
\end{prop}

Clearly, the following theorem holds.
\begin{thm}
$\psi_\E(N) = \sum_{i = 1}^{j} \psi_{\E, i}(N)$
\end{thm}

We consider the following method:
Given an equation $\E$ in $j$ variables, let $\psi_{\E, i}(N)$ and $\psi_{\E}(N)$ be defined as above. For each solution $X = (x_1, x_2, \ldots, x_j)$, clearly, there must exist at least one pair $i, j$ such that $x_i \neq x_j$. Otherwise, the trivial solution $(1, 1, \ldots, 1)$ would be a solution and $R_c(\E) = 1$. 

Randomly assign each element of the interval $[1, N]$ to an element of $[1, c]$. For any solution $X$, let $E_X$ be the event that $X$ is monochromatic under this random coloring. Let $Pr(E)$ 
denote the probability of event $E$ occurring. Note the following trivial bounds on $Pr(E_X)$:

\begin{thm}
$$\frac{1}{c^{j-1}} \leq Pr(E_X) \leq \frac{1}{c}.$$
\end{thm}
\begin{proof}
Let $X = (x_1, x_2, \ldots, x_j)$. As before, we discount the trivial solution $x_l = x_k$ for all $l, k \in [1, j]$. Now, for define $i$ as the number of distinct $x_k \in X$. Clearly, $2 \leq i \leq j$. The probability of $X$ being monochromatic is $\frac{c}{c^i} = \frac{1}{c^{i-1}}$. The bounds on $i$ give the result.
\end{proof}

Let $i$ be the number of distinct $x_k \in X$, then note that the proof of the above theorem gives
$$
Pr(E_X) = \frac{1}{c^{i-1}}.
$$

With $\psi_{\E, i}(N)$ and $\psi_{\E}(N)$ as defined above, we turn back to the random coloring of $[1, N]$. Let $E_s$ be the event that the coloring contains a monochromatic solution. We aim to show that $Pr(E_s) < 1$, implying the existence of a coloring of $[1, N]$ lacking a monochromatic solution. Let $Pr(E_i)$ be the probability that a randomly selected solution from the set of all solutions to $\E$ in the interval $[1, N]$ contains $i$ distinct $x_k$. It is not hard to see that $Pr(E_s)$ is simply:
$$
Pr(E_s) = \sum_{i = 1}^{j} Pr(E_i)\frac{1}{c^{i-1}}
$$
 
Now, 
$$
Pr(E_i) = \frac{\psi_{\E, i}(N)}{\psi_\E(N)}
$$
so we have

$$
Pr(E_s) = \sum_{i = 1}^{j} \frac{\psi_{\E, i}(N)}{\psi_\E(N)c^{i-1}}.
$$

Recall that $Pr(E_s) < 1$ implies the existence of a coloring of $[1, N]$ without a solution to equation $\E$. Thus, we have the following theorem.

\begin{thm}
Given an equation $\E$ in $j$ variables, $\psi_{\E, i}(N), \psi_{\E}(N)$, and $c$, $R_c(\E) > N$, where N satisfies
$$
\sum_{i = 1}^{j} \frac{\psi_{\E, i}(N)}{\psi_\E(N)c^{i-1}} < 1.
$$

or the equivalent:
$$
\sum_{i = 1}^j \psi_{\E, i}(N)c^{N-i+1} < c^N.
$$
\end{thm}

We give an example of this method's application to the equation $x - y = az$.

\begin{thm}
$R_c(x-y = bz) > N$, where $N$ satisfies $$\frac{N(c-1)}{b+1} + \frac{N(N+b)}{2b} < c^2.$$
\end{thm}

\begin{proof}
It is fairly easy to see that $$\psi_\E(N) = \frac{b(k)(k+1)}{2},$$ with $k = \floor{\frac{N}{b}}$. 
Similarly, we have $$\psi_{\E, 2}(N) = \floor{\frac{N}{b+1}}, \psi_{\E, 3}(N) = \psi_\E(N) - \floor{\frac{N}{b+1}}. $$ 
Direct application of the method outlined above, using $$
\sum_{i = 1}^j \psi_{\E, i}(N)c^{N-i+1} < c^N. 
$$
From $\frac{N}k \geq \floor{\frac{N}{k}}$, we have
$$
\frac{b(\frac{N}{b})(\frac{N}{b} + 1)}{2} = \frac{N(N+b)}{2b}.
$$

Then, after substituting, 
$$
\sum_{i = 1}^j \psi_{\E, i}(N)c^{N-i+1} \leq c^{N-1}\frac{N}{b+1} + c^{N-2}\frac{N(N+b)}{2b} - c^{N-2}\frac{N}{b+1}.
$$

So, after simplifying, if 
$$
\frac{N(c-1)}{b+1} + \frac{N(N+b)}{2b} < c^2,
$$
by Theorem 5.5, we have the result. 
\end{proof}

\begin{cor}
$$R_c(x-y = bz) > \frac{\sqrt{(b+c-1)^2 + 8c^2(b^2+b)} -b -c +1}{2}. $$
\end{cor}
\begin{proof}
This follows directly from applying the quadratic formula to the above theorem. 
\end{proof}

Note that the use of this method is only dependent upon finding the functions $\psi_{N, i}(\E)$ for families of equations. In individual equations, this method has the potential to be extremely versatile, as the number of solutions to a particular equation in $[1, n]$ in many cases is a simple computation.

We present the following results on $\psi_{\E, i}(N)$ for certain families of equations, 
whose proofs are left to the reader. Note that $\psi_{\E, 1}(N) = 0$ for $a(x-y) = bz$ and $x + ay =abz$. 
\begin{thm}
Let $\E = a(x-y) = bz$, $$\psi_\E(N) = \displaystyle\sum\limits_{i=b+1}^N\floor{\frac{i-1}{b}},$$
$$
\psi_{\E, 2}(N) = \floor{\frac{N}{a+b}} + \floor{\frac{N}{a}}, \psi_{\E, 3}(N) = \displaystyle\sum\limits_{i=b+1}^N\floor{\frac{i-1}{b}} - \psi_{\E, 2}(N).
$$
\end{thm}
\begin{cor}
$$R_c(a(x-y) = bz) > \frac{b(2a+b)(1-c)}{a(a+b)} + b\sqrt{(c-1)^2{\frac{(2a+b)}{a(a+b)}}^2 + \frac{2}{b}(\frac{b+3}{2} + c^2)} $$
\end{cor}

\begin{proof}
This follows from $$ \psi_\E(N) = \sum\limits_{i=b+1}^N\floor{\frac{i-1}{b}} \leq \sum\limits_{i=b+1}^N \frac{i-1}{b} = N - b - 2 + \frac{(N-b-1)(N-b)}{2b},$$ $$ \psi_{\E, 2} = \floor{\frac{N}{a+b}} + \floor{\frac{N}{a}} \leq \frac{N}{a+b} + \frac{N}{a},$$ $$\psi_{\E, 3} = \displaystyle\sum\limits_{i=b+1}^N\floor{\frac{i-1}{b}} - \psi_{E, 2}(N) \leq N - b - 2 + \frac{N-b-1}{2b} - \frac{N}{a+b} - \frac{N}{a} $$ and application of the quadratic formula to the expression generated by application of Theorem 5.4.
\end{proof}

We will obtain better lower bounds on 
$R_c(a(x-y) = bz)$, by using the Lovasz Local Lemma, in Theorem~\ref{th:LLL}.

Note that it is not necessary to find the exact value of $\psi_{\E, i}(N)$, upper bounds will suffice for application of the method.

\begin{lem}
Let $\E = x + ay = abz$, with $a, b \geq 2$, $$\psi_\E(N) < \floor{\frac{N-1}{b}}\floor{\frac{N}{a}},$$
$$
\psi_{\E, 2}(N) = 
\begin{cases}
\floor{\frac{N}{a(b-1)}} + \floor{\frac{N}{ab-1}} + \floor{\frac{N(a+1)}{ab}} & \text{when $b \equiv 0\bmod{(a+1)}$} \\
\floor{\frac{N}{a(b-1)}} + \floor{\frac{N}{ab-1}} + \floor{\frac{N}{ab}} & \text{otherwise}
\end{cases}
$$

$$
\psi_{\E, 3}(N) < \floor{\frac{N-1}{b}}\floor{\frac{N}{a}} - \psi_{\E, 2}.
$$
\end{lem}
This bound can be used to find bounds on the Rado numbers of equations of the form $x + ay = abz$ with the method presented above. 

Now, we find the corresponding $\psi_\E(N)$ of an equation in an arbitrary number of variables.
\begin{thm}
Let $\E = x_1 + x_2 + \ldots x_j = x_{j+1} + x_{j+2} + \ldots x_k$, $$\psi_\E(N) = \displaystyle\sum\limits_{i=j}^N\binom{i-1}{j-1}\binom{i-1}{j-k-1}.$$
\end{thm}

\begin{proof}
Assume $j \geq k$, the alternative case is equivalent by symmetry. From a combinatorial argument, we have that the number of solutions in positive integers to $x_1 + x_2 + \ldots + x_j = m$ is ${m-1}\choose{j-1}$. Similarly, the number of solutions to $x_{j+1} + x_{j+2} + \cdots + x_{k} = m$ is ${m-1}\choose{j-k-1}$. So, the number of solutions to $x_1 + x_2 + \cdots + x_j  = x_{j+1} + x_{j+2} + \cdots + x_k$ is $\displaystyle\sum\limits_{i=j}^N$ $\binom{i-1}{j-1}\binom{i-1}{j-k-1}$.
\end{proof}
From this result, we can find bounds on the Rado numbers of equations of the form of $\E$ in arbitrary number of colors and variables, as long as we can find the corresponding $\psi_{\E, i}(N)$ for each equation.

Note that using this method to bound Rado numbers depends upon finding a closed form expression for the functions $\psi(N)_{\E, i}$, which may become a difficult problem. However, given $N$ and equation $\E$, computing the number of solutions to $\E$ in the interval $[1, N]$ is not a difficult computation. Additionally, determining the number of distinct values contained within each solution is relatively simple. Thus, we present a simple algorithmic approach to our method outlined above that can be used to bound Rado numbers.

Briefly, the algorithm computes the values of $\psi_{\E, i}(N)$ and $\psi_{\E, i}(N + 1)$ by counting the number of distinct $x_k$ in each solution in $[1, N]$ and $[1, N+1]$ respectively. If 
$$
\sum_{i = 1}^j \psi_{\E, 1}(n)c^{n-i+1} < c^n, \text{and} \sum_{i = 1}^j \psi_{\E, 1}(n+1)c^{n-i+2} \geq c^{n+1},
$$
$N$ must be the maximal integer that satisfies the inequality, so $R_c(\E) > N$.
 \\

\ttfamily

\textbf{Input}: $E = a_1x_1 + a_2x_2 + \cdots a_jx_j = 0$ AND $c$ 

\indent Set $n = k$

\indent \textbf{while} TRUE

\indent\indent Find Solutions in $[1, n]$ and $[1, n+1]$

\indent\indent Count Solutions with $i$ distinct $x_k$; Assign values to  $\psi_{\E, i}(N)$, $\psi_{\E, i}(N+1)$ 

\indent\indent \textbf{if} $\sum_{i = 1}^j \psi_{\E, 1}(n)c^{n-i+1} < c^n$ AND $\sum_{i = 1}^j \psi_{\E, 1}(n+1)c^{n-i+2} \geq c^{n+1}$

\indent\indent\indent \textbf{Return} $n$

\indent\indent \textbf{else} Increment $n$

\rmfamily

\begin{subsection}{Using the Lov\`asz Local Lemma}
Note that overcounting, by assuming that each solution to $\E$ is independent of the others, will make the bounds from this method fairly loose. To deal with the dependence amongst solutions sets, we will utilize the concept of a dependency graph and a theorem known as the Lov\`asz Local Lemma, used extensively in the probabilistic method.

We use the following definition of a dependency graph for an equation $\E$ in the interval $[1, N]$.
\begin{prop}[Dependency graph]
Given an equation $\E$ in $j$ variables with $l$ solutions in $[1, N]$, let $X_i$ represent the $i$th $j$-tuple that satisfies $\E$. The dependency graph $G$ on the solution sets $X_i$ is constructed as follows: for every $X_k$ if $X_k \cap X_n \notin {\emptyset}$ for $n \in [1, N]$, $(k, n) \in E(G)$.
\end{prop}

The following theorem is a sieve method used in instances where many of the events in a probability space are independent, but their does exist some dependency between distinct events. For proof, consult \cite{Lovasz}.
\begin{thm}[Lov\`{a}sz Local Lemma]
Let $A_1, A_2, \ldots, A_k$ be events in a probability space $\Omega$ such that $Pr[A_i] \leq p < 1$. Define $d_i$ as the the number of events that are pairwise dependent to $A_i$, for $i \in [1, k]$, and $d = max(d_i)$. If $ep(d + 1) < 1$, there is a probability > 0 that none of the events $A_i$ occur.
\end{thm}

In order to use the Lov\`asz Local Lemma to bound Rado numbers, we need to define some terms

\begin{prop}~
\begin{enumerate}
\item
$\upsilon_i$ is the degree of $G_N(\E)$, the dependency graph of equation $\E$ over $[1, N]$.
\item
$\phi_N(\E)$ is the be $max(\upsilon_i)$, the maximum degree of $G_N(\E)$. 
Note that $\phi_N(\E) = d$ in the Lovasz Local Lemma. 
$\phi_N(\E)$ can also be considered as the maximum number of dependent solutions 
\end{enumerate}
\end{prop}

Note that $\phi_N(\E)$ can also be considered as the maximum number of dependent solutions 
in $[N]$. If we can find $\phi_N(\E)$ as a function of $N$, 
we can use the Lovasz Local Lemma to determine lower bounds on the $c$-color Rado number of equation $\E$. 
We describe this method in the following theorem.

\begin{thm}\label{th:lower}
Let $\phi_N(\E)$ be as defined above, for equation $\E$ in $i$ variables. Then, applying the Lovasz Local Lemma, $R_c(\E) > N$, where $N$ satisfies $$\phi_N(\E) + 1 < \frac{c^i}{e}.$$
\end{thm}
This follows from the obvious fact that a random $c$-coloring of an $i$-tuple will be monochromatic with probability $\frac{1}{c^i}$.

Following, we give $\phi_N(\E)$ of some classes of equations. The proofs are fairly simple, and are left to the reader. 

\begin{lem}
Let $a<b$ and let $\E$ be  $a(x - y) = bz$. Then 
$$\phi_N(\E) = \floor{\frac{N-1}{b}} + 2\floor{\frac{N-b-1}{b}} + N - b + 1.$$
\end{lem}
Note that for $a < b$, $\phi_N(a(x-y) = bz)$ does not depend upon the value of $a$.
\begin{lem}
Let $E = a(x-y) = bz$, for $a > b$, $$\phi_N(\E) = 2\floor{\frac{N}{a}} + N - b + 1.$$
\end{lem}

We show the application of this method by considering the case $b = 2$.
\begin{thm}
$$R_c(x-y=2z) \geq  \frac{2c^3}{5e} + \frac{7}{5}.$$ $$R_c(a(x-y) = 2z) \geq \frac{ac^3}{(a+2)e}$$ for $a > 2$.
\end{thm}

\begin{proof}
For the case $a = 1$, $$\floor{\frac{N-1}{2}} + 2\floor{\frac{N-3}{2}} + N - 1 \leq \frac{N-1}{2} + \frac{2(N-3)}{2} + N - 1 =  \frac{5N-7}{2} + 1.$$ Letting $d = \frac{5N-7}{2} + 1$, and from $p = \frac{1}{c^3}$, and application of the Lovasz Local Lemma, we have the result. 

For the case $a > 2$, $$2\floor{\frac{N}{a}} + N - 1 \leq \frac{2N}{a} + N - 1 = \frac{(a+2)N}{a} -1.$$ Letting this equal $d$, and $p = \frac{1}{c^3}$ as above, and applying the Lov\`asz Local Lemma, we have the result. 
\end{proof}

The proof of the following is similar.
\begin{thm}\label{th:LLL}
 For $a < b$, $$R_c(a(x-y)=bz) \geq \frac{bc^3}{e(b+3)} + \frac{2b+3}{b+3}.$$
\end{thm}

The lower bound on $R_c(\E)$ from Theorem 5.14 applies to any equation $\E$,
not just regular ones. Of course, if $R_c(\E)$ does not exist then the
bound is not useful. Nevertheless, we present a lower bound on $R_c(\E)$
for an $\E$ that is not necessarily regular (though it may be $c$-regular for some values of $c$).

\begin{thm}
Let $E = x + ay = abz$, $$\phi_N(\E) = 3\floor{\frac{N+1}{b}}.$$
\end{thm}
Note that determining $\psi_N(\E)$ is a much easier task than $\phi_N(\E)$, but utilizing the Lov\`asz Local Lemma gives much better bounds on $R_c(\E)$. We believe that refining these arguments, and defining $\psi_N(\E)$ and $\phi_N(\E)$ for many more classes of equations, will help us better understand the behavior of $R_c(\E)$ for $c > 2$.

\end{subsection}

\end{section}

\begin{section}{Conjectures on Rado Numbers}
Using empirical results from Rado numbers computed using our algorithm, we propose the following conjectures:
\subsection{$x + qy = q^2z$}
Here we provide a conjecture on Rado numbers of equations of the form $x + qy = q^2z$ with $q \in \mathbb{N}$. 

\begin{conj}
$$
R(x + qy = q^2z) = 
\begin{cases}
\frac{q^3}{2} & \text{if $q \equiv 0 \bmod{2}$} \\
\frac{q^3 + q}{2} & \text{if $q \equiv 1 \bmod{2}$}
\end{cases}
$$
\end{conj}

\subsection{4-Variable Conjuncture}
Here, a conjecture is posed regarding the extension of the three variable equation: $x + ay = 2az$, whose Rado number was determined in Section 3.2 to be $a^2$. 

\begin{conj}
$R(x + a(x+y) = 2aw) = a^2$
\end{conj}

Note that the lower bound is established by the theorem in Section 3.2.
\end{section}

\begin{section}{Empirical Results}
Tables of computed 2 and 3-color Rado numbers are can be found in the appendix. We include an extension of the table of computed 2-color Rado numbers presented in Meyers and Robertson's paper \cite{2color}, as well as the first published results on $3$-color Rado numbers. 

In the appendix, we also include our formulation of the problem of determining Rado numbers as a Boolean Satisfiability problem - the motivation for our own algorithm which was loosely based off of a randomized SAT algorithm. It is our belief that utilizing powerful SAT solvers will provide a new route to the computation of Rado numbers for equations in a large number of variables and colors. It will be of interest to see if industrial SAT solvers will be able to make headway in this regard.

The present paper focuses on the efficient computation of Rado numbers, and presents several bounds on Rado numbers for a few classes of equations. A new proof of Rado's theorem is given, yielding better bounds on Rado numbers in certain cases. Finally, a probabilistic method of constructing lower bounds of Rado numbers is given. Significant progress was made in the area of computation of Rado numbers - the first comprehensive list of 3-color Rado numbers has been developed. Additionally, further exploration of our proof of Rado's theorem may give rise to new bounds. The probabilistic method bounds represent, to the best of our knowledge, the first expressions giving bounds on Rado numbers in greater than 2 colors. 
\end{section}

\newpage

\appendix

\section{Algorithm}
\subsection{Our Algorithm}

Previous attempts at the computation of Rado numbers have relied on backtracking methods - computationally inefficient in the analysis of Rado numbers in a large number of variables or colors. We developed a simple probabilistic algorithm to compute Rado numbers. Note that our algorithm is a variant of Sch\"oning's algorithm for the Boolean satisfiability (SAT) problem. In the following section we elucidate the method of converting the problem computation of Rado numbers to the SAT problem. Briefly, our algorithm functions as follows. 

\ttfamily

\textbf{Input equation}: $a_1x_1 + a_2x_2 + \ldots a_kx_k = 0$, $c$

\indent Set $N = k$

\indent \textbf{\emph{RandomColor}}: Assigns a random color, in $[1, c]$, to each number in $[1, N]$.

\indent \indent \textbf{\emph{FindSolutions}}: Returns integral solutions in $[1, N]$.

\indent \indent \textbf{For 1 to $3N\sqrt{N}(\frac{4}{3})^N$}

\indent \indent \indent \textbf{\emph{Solve}}: Iterates through each solution, finds the first monochromatic \indent \indent \indent solution $(x_1, x_2, \ldots, x_j)$. If no monochromatic solution is found, $N$ is \indent \indent \indent incremented, assigned a random color. Call \emph{FindSolutions}.

\indent \indent \indent \textbf{\emph{ChangeColor}}: A number in $[1, j]$ is randomly chosen, and the color of $x_j$ \indent \indent \indent is randomly changed. Solve is repeated.

\indent \indent Return $N$

\rmfamily
Due to the probabilistic nature of the solution, the algorithm is able to quickly find "bad" $N$'s, those containing colorings without monochromatic solutions, and move on. However, because the algorithm is probabilistic, there exists an error - or the probability of the algorithm missing a valid coloring. For the number of steps chosen in our algorithm, the error is bounded by $e^{-20}$ (where $e$ is the base of the natural logarithm). Consult Sch\"oning's original paper \cite{Schoning} for the proof of this result. 
 
This algorithm was inspired by the $k$-SAT algorithm developed by Sch\"oning in 1999\cite{Schoning}. We note in the appendix how the problem of the computation of Rado numbers can be easily transformed to an instance of $k$-SAT.

For $c = 2$, our algorithm runs in $poly(n)(\frac{4}{3})^n$ time; however, for $c \geq 3$, note that the algorithm's running time is bounded by $(poly(n)(\frac{4}{3})^n)^c$. 

In the appendix, we present in more detail the expression of the Rado problem as an instance of SAT. Although in the current paper no SAT Solvers were utilized, it is our belief that it will be possible to compute the Rado numbers of complex equations in a large number of colors with industrial SAT Solvers, utilizing the methods elucidated below. 

\section{Computing Rado Numbers with SAT}

\subsubsection{Overview of the Boolean Satisfiability Problem}
Given a boolean formula, the boolean satisfiability problem (SAT) is to determine the values of the boolean variables within the formula that will make the expression evaluate to True. The SAT problem was the first problem proved to be NP-Complete, and, as such, there are no known methods for efficiently solving SAT on a large scale. However, the SAT problem is easy for small inputs, and many algorithms have been developed to decrease the running time of the algorithm. We introduce some notation: Conjunctive Normal Form (CNF) refers to a boolean expression with literals or their negations separated by OR's within clauses separated by AND's. An example is given below:
\begin {equation*}
(a_1 \vee a_2 \vee a_3) \wedge (b_1 \vee \neg b_2 \vee b_3) \ldots
\end{equation*}

The $k$-SAT problem takes a boolean expression in CNF form as its input, with at most $k$ literals in each clause, and outputs an assignment of values to the literals, that evaluates the expression to TRUE. If such an assignment does not exist, the algorithm will output FALSE. Many problems have been reduced to instances of the $k$-SAT problem, and in the following sections we show how the problem of finding 2-color Rado numbers can be posed as an instance of $k$-SAT. By implementing efficient algorithms that have been developed for general case SAT problems, we believe it will be possible to greatly decrease the run time for the computation of Rado numbers.

\subsection{Sch\"oning's Algorithm and Rado as a SAT Problem}

Recall that our algorithm was loosely based off Sch\"oning's algorithm for the Boolean Satisfiability problem. Here we present an outline of Sch\"onings probabilistic SAT algorithm \cite{Schoning}, and describe deviations from this algorithm in our implementation. 

Briefly, the algorithm randomly assigns values to the Boolean variables present in the expression. It then finds the first clause in the expression that evaluates to false, and randomly picks a variable in the clause and changes its value. This process is repeated $3n\sqrt{n}(\frac{4}{3})^n$ times and if no solution to the SAT problem is solved, the algorithm returns UNSATISFIABLE. 

In our algorithm, rather than considering an initial random coloring of $[1, N]$ after $N$ is incremented, the previous coloring of $[1, N-1]$ that contained no monochromatic solutions is carried over. Clearly, many colorings that contain no monochromatic solutions in $[1, N-1]$ will also contain no monochromatic solutions in $[1, N]$, so this step reduces unnecessary searching in many cases.

Additionally, we have included a factor $k$ into the function specifying the number of steps: $k(3N\sqrt{N}(\frac{4}{3})^N)$, which can be adjusted to give varying degrees of error from the probabilistic algorithm. Note that for $k = 1$, the error - or probability of missing a valid coloring, is given by $e^{-20}$. Consult Sch\"oning's original paper \cite{Schoning} for the proof of this result.

\subsubsection{Finding Two-Color Rado Numbers of Three Variable Equations with 3-SAT}
Given an equation, let the solutions of the equation, in the interval $[1, N]$ be given by $(x_1, y_1, z_1), (x_2, y_2, z_2), \ldots (x_n, y_n, z_n)$. Let the colors be $\left\{0,1\right\}$, and define the color of $j$ to be the boolean variable $C_j$. 

The below expression will evaluate to true if and only if the current coloring contains no monochromatic solutions:
\begin{equation*}
(C_{x_{1}} \vee C_{y_{1}} \vee C_{z_{1}}) \wedge (\neg C_{x_{1}} \vee \neg C_{y_{1}} \vee \neg C_{z_{1}}) \wedge (C_{x_{2}} \vee C_{y_{2}} \vee C_{z_{2}}) \wedge (\neg C_{x_{2}} \vee \neg C_{y_{2}} \vee \neg C_{z_{2}}) \cdots
\end{equation*}

To see why, imagine a monochromatic solution $(x_i, y_i, z_i)$. The color is either 0 or 1, so one of $(C_{x_{i}} \vee C_{y_{i}} \vee C_{z_{i}})$ and $(\neg C_{x_{i}} \vee \neg C_{y_{i}} \vee \neg C_{z_{i}})$ must evaluate to false. Becuase the clauses are in CNF, if one clause evaluates to false, the expression must as well. 

\subsubsection{Extension to the $k$-Variable Case}
Just as 3-SAT can be used to compute the Rado number for an equation in three variables, the Rado number of a general equation in k variables can be computed with $k$-SAT. The expression for the 3-SAT case contains one literal for each variable in a clause. Extending this to the general case, we see that each clause should contain a literal for each variable in the equation. Below is an example, the particular solution is given by $(x_1, x_2, x_3, x_4, \ldots x_k)$. 

\begin{equation*}
(C_{x_{1}} \vee C_{x_{2}} \vee C_{x_{3}} \vee C_{x_{4}} \cdots \vee C_{x{k}}) \wedge (\neg C_{x_{1}} \vee \neg C_{x_{2}} \vee \neg C_{x_{3}} \vee \neg C_{x_{v}} \cdots \vee \neg C_{x_{k}}) \cdots
\end{equation*}

\subsubsection{Extension to the $c$-Color Case}
We now show how $k$-SAT can be used to solve for the $c$-Color Rado number of an equation. Define $i_l = 1$ if $COL(i) = l$, for $0 \leq l \leq c-1$. Clearly, only one $i_j$ can equal 1 for a given coloring. Then, the following expression will evaluate to true if and only if there exists no monochromatic colorings under the current coloring.
Let the solutions to the equation be given by $(x_1, x_2, x_3, \ldots x_k)$. Then, the following expression evaluates to true if and only if the current coloring does not contain any monochromatic solutions.
\begin{equation*}
(\neg{x_{1_1}} \vee \neg{x_{2_1}} \vee \neg{x_{3_1}} \cdots \vee \neg{x_{k_1}}) \wedge (\neg{x_{1_2}} \vee \neg{x_{2_2}} \vee \neg{x_{3_2}} \cdots \vee \neg{x_{k_2}}) \cdots \wedge (\neg{x_{1_c}} \vee \neg{x_{2_c}} \vee \neg{x_{3_c}} \ldots \vee \neg{x_{k_c}})
\end{equation*}

\newpage

\section{Computed Rado Numbers}
In this section we present some of the Rado numbers that were computed using our algorithm. Note that due to the probabilistic nature of the algorithm, the probability of missing a valid coloring is bounded by $e^{-20}$. Consult Sch\"oning's original paper \cite{Schoning} for the proof of this error bound. 

\subsection{3-Color Rado Numbers}
Here we present a table of computed lower bounds for 3-Color Rado numbers of equations of the form $(a(x-y) = bz)$ using our algorithm- to the best of our knowledge no previous studies have presented such a table for $R_3(\E)$.

\vspace{1cm}

\resizebox{12cm}{!}{

\begin{tabular}{r| r@{\hspace{1cm}} c@{\hspace{1cm}} c@{\hspace{1cm}} c }

 $R_3(a(x-y) = bz)$ &      $a$ = 1 &          2 &          3 &          4 \\
\hline
     $b$ = 1 &         14 &         14 &         27 &         57 \\

         2 &         42 &         14 &         31 &         14 \\

         3 &         78 &         56 &         14 &         64 \\

         4 &         94 &         43 &         67 &         14 \\

         5 &        142 &        108 &         85 &         81 \\

         6 &        161 &         80 &         42 &         54 \\

         7 &        178 &        157 &        136 &        128 \\

         8 &        193 &        127 &        157 &         43 \\

         9 &        213 &        190 &         80 &        163 \\

        10 &        237 &        142 &        202 &         98 \\

        11 &        247 &        227 &        211 &        204 \\

        12 &        258 &        156 &        120 &         78 \\

        13 &        291 &        255 &        250 &        244 \\

        14 &        299 &        178 &        267 &        154 \\

        15 &        318 &        302 &        140 &        278 \\

        16 &        348 &        197 &        309 &        125 \\

        17 &        358 &        334 &        317 &        312 \\

        18 &        380 &        216 &        167 &        192 \\

        19 &        416 &        370 &        372 &        351 \\

        20 &        416 &        243 &        375 &        148 \\

        21 &        440 &        410 &        179 &        367 \\

        22 &        461 &        252 &        411 &        230 \\

        23 &        462 &        439 &        424 &        418 \\

        24 &        485 &        276 &        196 &        155 \\

        25 &        500 &        495 &        446 &        438 \\

\end{tabular}  

}

\newpage

\subsection{2-Color Rado Numbers}
Here we present a table of computed 2-Color Rado numbers for equations of the form \\ $2(x-y) + az = bw$ - an extension of some of the results in  Meyers and Robertson's \\ paper \cite{2color}. 

\vspace{0.5cm}
\center
\resizebox{18.5cm}{!}{

\begin{tabular}{r|rrrrrrrrrrrrrrrr}

  $2(x-y) + az = bw$&      $a$ = 1 &          2 &          3 &          4 &          5 &          6 &          7 &          8 &          9 &         10 &         11 &         12 &         13 &         14 &         15 &         16 \\
\hline
    $b$ = 1 &         29 &         76 &         86 &        106 &        119 &        145 &        156 &        190 &        201 &        238 &        258 &        290 &        326 &        361 &        389 &        430 \\

         2 &          8 &         11 &         23 &         19 &         40 &         29 &         65 &         41 &         92 &         55 &        123 &         71 &        158 &         89 &        191 &        108 \\

         3 &          5 &          4 &          9 &         12 &         10 &         16 &         21 &         18 &         25 &         28 &         29 &         32 &         41 &         36 &         48 &         56 \\

         4 &          4 &          4 &          8 &          5 &         12 &          8 &         22 &          9 &         29 &         15 &         45 &         17 &         54 &         23 &         76 &         25 \\

         5 &          1 &          8 &          5 &          4 &          9 &          6 &          7 &         12 &         13 &         12 &         11 &         16 &         19 &         16 &         25 &         16 \\

         6 &          4 &          1 &          4 &          4 &          7 &          5 &          9 &          4 &         15 &          8 &         22 &          9 &         26 &         10 &         32 &         15 \\

         7 &          3 &          8 &          1 &          6 &          5 &          4 &         12 &          6 &         10 &          8 &         12 &         12 &         14 &         16 &         13 &         14 \\

         8 &          4 &          4 &          6 &          1 &          5 &          4 &          7 &          4 &          9 &          6 &         13 &          5 &         15 &          7 &         22 &          8 \\

         9 &          5 &         10 &          9 &          7 &          1 &          9 &          5 &          6 &         15 &          9 &         13 &          9 &         13 &          8 &         15 &         16 \\

        10 &         12 &          9 &          6 &          4 &          6 &          1 &          5 &          5 &          7 &          8 &         12 &          6 &         11 &          6 &         15 &          6 \\

        11 &          7 &         16 &          7 &          9 &          7 &          8 &          1 &          8 &          5 &          8 &         18 &         10 &         10 &         11 &         16 &         10 \\

        12 &          8 &          4 &          9 &          3 &          8 &          4 &          7 &          1 &          9 &          5 &          8 &          6 &          9 &          7 &         12 &          9 \\

        13 &          9 &         23 &          6 &         14 &          6 &         10 &          8 &          8 &          1 &         10 &          5 &          8 &         21 &         10 &         11 &         12 \\

        14 &         10 &         10 &         12 &          4 &          7 &          3 &         10 &          4 &          6 &          1 &          5 &          5 &          8 &         11 &         11 &          8 \\

        15 &          9 &         27 &         10 &         16 &          5 &          9 &          6 &         10 &          9 &         10 &          1 &         10 &          5 &          8 &         24 &         12 \\

        16 &         13 &         12 &         12 &          5 &          9 &          4 &          8 &          3 &          9 &          4 &          6 &          1 &          8 &          5 &         11 &          8 \\

        17 &         14 &         35 &         10 &         24 &         10 &         14 &         10 &          9 &          6 &         10 &          8 &         11 &          1 &          8 &          5 &          8 \\

        18 &         15 &         14 &         16 &          7 &         12 &          9 &          9 &          6 &         12 &          6 &         10 &          9 &          9 &          1 &          9 &          5 \\

        19 &         16 &         45 &         17 &         26 &          7 &         16 &          8 &         15 &          8 &          9 &          6 &         10 &          9 &         11 &          1 &         10 \\

        20 &         16 &         16 &         18 &          8 &         13 &          6 &         11 &          5 &         12 &          8 &         10 &          5 &         10 &          4 &         15 &          1 \\

        21 &         26 &         49 &         15 &         32 &          7 &         22 &         11 &         14 &          9 &         11 &         10 &         14 &         10 &         10 &          9 &         12 \\

        22 &         27 &         18 &         22 &          9 &         16 &          7 &         12 &          6 &         11 &          5 &         14 &          8 &         12 &          5 &         11 &          6 \\

        23 &         28 &         62 &         20 &         35 &         13 &         24 &          9 &         16 &         11 &         14 &          5 &         12 &         10 &         14 &         10 &         10 \\

        24 &         30 &         20 &         28 &          9 &         16 &         10 &         14 &          4 &         12 &          6 &         14 &          5 &         12 &          8 &         12 &          5 \\

        25 &         31 &         61 &         17 &         44 &         18 &         26 &         10 &         21 &          8 &         15 &         11 &         13 &          6 &         13 &         15 &         14 \\

        26 &         33 &         22 &         30 &         17 &         20 &          9 &         16 &         10 &         14 &          4 &         12 &         11 &         18 &          6 &         14 &          8 \\

        27 &         34 &         82 &         27 &         48 &         18 &         34 &         11 &         23 &         27 &         17 &         12 &         14 &         11 &         12 &         15 &          9 \\

        28 &         35 &         24 &         16 &         18 &         22 &         10 &         19 &          8 &         16 &         10 &         14 &          4 &         14 &         11 &         13 &          5 \\

        29 &         37 &         73 &         21 &         59 &         18 &         36 &         12 &         25 &         16 &         22 &         14 &         17 &         12 &         15 &         11 &         12 \\

        30 &         38 &         26 &         39 &         20 &         28 &          9 &         19 &         11 &         18 &         11 &         16 &          9 &         14 &          7 &         20 &         10 \\

\end{tabular}  
}
\newpage

\begin{flushleft}
Continued table of computed lower bounds for 2-color Rado numbers of equations of the form $2(x-y) + az = bw$.
\end{flushleft}

\center
\vspace{0.7cm}
\resizebox{18.5cm}{!}{

\begin{tabular}{r|rrrrrrrrrrrrrrrr}

 $2(x-y) + az = bw$ &      $a$ = 1 &          2 &          3 &          4 &          5 &          6 &          7 &          8 &          9 &         10 &         11 &         12 &         13 &         14 &         15 &         16 \\
\hline
    $b$ = 31 &         44 &        102 &         27 &         63 &         22 &         45 &         11 &         32 &         10 &         24 &         16 &         19 &         15 &         19 &         12 &         20 \\

        32 &         45 &         38 &         46 &         21 &         29 &         13 &         20 &         10 &         20 &         11 &         14 &         10 &         16 &         10 &         16 &          7 \\

        33 &         47 &        103 &         24 &         68 &         19 &         46 &         12 &         34 &         18 &         24 &         17 &         23 &         15 &         22 &         15 &         18 \\

        34 &         48 &         40 &         44 &         23 &         31 &         14 &         22 &         11 &         18 &         11 &         20 &         11 &         16 &         12 &         16 &         10 \\

        35 &         49 &        108 &         25 &         74 &         24 &         51 &         14 &         36 &         12 &         25 &         19 &         24 &         17 &         21 &         15 &         21 \\

        36 &         58 &         43 &         57 &         24 &         32 &         15 &         27 &          9 &         23 &         13 &         18 &         15 &         20 &         12 &         18 &          9 \\

        37 &         60 &        116 &         32 &         93 &         21 &         60 &         14 &         46 &         20 &         32 &         12 &         24 &         18 &         24 &         17 &         21 \\

        38 &         62 &         48 &         66 &         26 &         38 &         16 &         28 &         13 &         24 &         12 &         18 &         13 &         16 &         13 &         16 &         11 \\

        39 &         63 &        129 &         57 &         86 &         36 &         63 &         21 &         46 &         18 &         34 &         15 &         25 &         26 &         23 &         16 &         23 \\

        40 &         64 &         50 &         58 &         27 &         42 &         16 &         30 &         14 &         24 &         14 &         19 &         12 &         22 &         13 &         20 &         12 \\

        41 &         74 &        137 &         48 &         84 &         25 &         74 &         23 &         50 &         15 &         36 &         18 &         32 &         21 &         27 &         20 &         26 \\

        42 &         76 &         53 &         72 &         32 &         48 &         26 &         32 &         14 &         28 &         10 &         21 &         16 &         21 &         14 &         20 &         16 \\

        43 &         78 &        144 &         56 &        119 &         39 &         78 &         24 &         58 &         16 &         44 &         15 &         34 &         16 &         27 &         22 &         27 \\

        44 &         80 &         59 &         82 &         33 &         50 &         27 &         32 &         15 &         29 &         14 &         28 &         10 &         24 &         15 &         19 &         17 \\

        45 &         81 &        152 &         52 &         97 &         25 &         81 &         27 &         54 &         18 &         46 &         18 &         34 &         17 &         30 &         45 &         29 \\

        46 &        102 &         62 &         73 &         35 &         57 &         28 &         38 &         16 &         30 &         15 &         29 &         12 &         22 &         10 &         20 &         15 \\

        47 &        104 &        193 &         54 &        124 &         42 &         79 &         29 &         63 &         22 &         48 &         17 &         36 &         20 &         32 &         19 &         27 \\

        48 &        106 &         64 &         96 &         36 &         64 &         30 &         44 &         17 &         30 &         14 &         28 &          9 &         23 &         18 &         24 &         16 \\

        49 &        108 &        201 &         57 &        135 &         34 &         89 &         36 &         74 &         19 &         50 &         20 &         43 &         12 &         35 &         18 &         29 \\

        50 &        111 &         84 &        108 &         44 &         54 &         31 &         46 &         21 &         32 &         17 &         32 &         16 &         25 &          9 &         28 &         18 \\

        51 &        115 &        213 &         62 &        140 &         41 &         86 &         33 &         77 &         21 &         60 &         24 &         45 &         20 &         35 &         21 &         33 \\

        52 &        117 &         87 &         90 &         46 &         70 &         33 &         48 &         22 &         32 &         15 &         34 &         13 &         27 &         16 &         24 &          9 \\

        53 &        120 &        224 &         64 &        146 &         42 &        104 &         35 &         80 &         21 &         54 &         22 &         47 &         20 &         37 &         20 &         34 \\

        54 &        122 &         91 &        121 &         48 &         72 &         34 &         50 &         23 &         38 &         16 &         30 &         18 &         29 &         11 &         26 &         19 \\

        55 &        124 &        235 &         71 &        151 &         50 &        113 &         31 &         92 &         22 &         56 &         22 &         49 &         22 &         44 &         24 &         36 \\

        56 &        135 &         98 &        116 &         49 &         82 &         35 &         56 &         24 &         44 &         16 &         30 &         14 &         28 &         18 &         27 &         12 \\

        57 &        137 &        246 &         74 &        147 &         45 &        118 &         34 &         81 &         23 &         72 &         33 &         58 &         21 &         46 &         21 &         35 \\

        58 &        140 &        102 &        108 &         66 &         91 &         37 &         64 &         25 &         46 &         21 &         32 &         15 &         33 &         19 &         29 &         17 \\

        59 &        142 &        257 &         77 &        159 &         47 &        109 &         39 &         92 &         24 &         56 &         22 &         60 &         26 &         46 &         20 &         36 \\

        60 &        144 &        105 &        129 &         68 &         86 &         38 &         66 &         25 &         47 &         22 &         34 &         15 &         35 &         14 &         31 &         18 \\

\end{tabular}  
}

\end{document}